\documentclass[12pt]{amsart}

\usepackage{amsmath}
\usepackage{enumerate}
\usepackage{xspace}
\newtheorem{theorem}{Theorem}[section]
\newtheorem{lemma}[theorem]{Lemma}
\newtheorem{corollary}[theorem]{Corollary}

\theoremstyle{definition}
\newtheorem{definition}[theorem]{Definition}
\newtheorem{convention}[theorem]{Convention}
\newtheorem{notation}[theorem]{Notation}

\newtheorem{remark}[theorem]{Remark}
\newtheorem{remarks}[theorem]{Remarks}

\makeatletter
\@namedef{subjclassname@2010}{%
  \textup{2010} Mathematics Subject Classification}
\makeatother

\def\uI{{\underline I}}
\def\uJ{{\underline J}}
\def\ux{{\underline x}}
\def\uz{{\underline z}}

\def\hf{{\hat{f}}}
\def\hI{{\hat{I}}}
\def\hX{{\hat{X}}}
\def\hmu{{\hat{\mu}}}
\def\hnu{{\hat{\nu}}}

\def\R{{\mathbb R}}
\def\N{{\mathbb N}}

\DeclareMathOperator{\Cl}{Cl}
\DeclareMathOperator{\Intt}{Int}
\DeclareMathOperator{\orb}{orb}

\newcommand{\thr}[1]{\left\langle #1 \right\rangle}
\newcommand{\thrn}[1]{\thr{#1_0, #1_1, #1_2, \dots}}
\newcommand{\cyl}[1]{\left[#1\right]}
\newcommand{\cyls}[1]{[#1]}
\newcommand{\cyln}[1]{\cyl{#1_0, #1_1, \dots, #1_n}}
\newcommand{\up}[1]{^{(#1)}}

\newcommand{\simple}{flat\xspace}
\newcommand{\Simple}{Flat\xspace}

\def\GR{{\mathcal{GR}}}

\def\raw{\rightarrow}

\def\ilim{{\varprojlim}}

\numberwithin{equation}{section}

\begin{document}


\baselineskip=17pt


\title[Typical path components]{Typical path components in tent
	map inverse limits}
\date{}

\author[P. Boyland]{Philip Boyland}
\address{Department of
    Mathematics\\University of Florida\\372 Little Hall\\Gainesville\\
    FL 32611-8105, USA}
\email{boyland@ufl.edu}
\author[A. de Carvalho]{Andr\'e de Carvalho}
\address{Departamento de
    Matem\'atica Aplicada\\ IME-USP\\ Rua do Mat\~ao 1010\\ Cidade
    Universit\'aria\\ 05508-090 S\~ao Paulo SP\\ Brazil}
\email{andre@ime.usp.br}
\author[T.Hall]{Toby Hall}
\address{Department of Mathematical Sciences\\ University of
    Liverpool\\ Liverpool L69 7ZL, UK}
\email{tobyhall@liv.ac.uk}



\begin{abstract}
	In the inverse limit $\hI_s$ of a tent map $f_s$ restricted to its core,
	the set $\GR$ of points whose path components are bi-infinite and bi-dense
	has full measure with respect to the measure induced on $\hI_s$ by the
	unique absolutely continuous invariant measure of $f_s$. With respect to
	topology, there is a dichotomy. When the parameter~$s$ is such that the
	critical orbit of~$f_s$ is not dense, $\GR$ contains a dense $G_\delta$
	set. In contrast, when the critical orbit of~$f_s$ is dense, the complement
	of $\GR$ contains a dense $G_\delta$ set.
\end{abstract}

\maketitle

\section{Introduction}

The inverse limits $\hI_s$ of tent maps $f_s$ restricted to their core
intervals have been the subject of intense investigation in dynamics and
topology. In dynamics they are models for
attractors~\cite{mis,bruin,bargemartin,ingrambook}, while in topology they are
studied for their intrinsic topological complexity. The main recent focus of
topological investigations has been the proof of the Ingram conjecture: for
different values of $s\in [\sqrt{2}, 2]$ the inverse limits are not
homeomorphic (for references see~\cite{Ingram}, which contains the final proof
for tent maps which are not restricted to their cores). Many distinguishing
properties have been discovered. For example, when the parameter $s$ is such
that critical orbit of $f_s$ is dense (a full measure, dense $G_\delta$ set of
parameters), theorems of Bruin and of Raines imply that the inverse limit
$\hI_s$ is nowhere locally the product of a Cantor set and an
interval~\cite{bruin,raines}. Perhaps more striking, for a dense, $G_\delta$
set of parameters~$s$, Barge, Brooks and Diamond~\cite{BBD} show that the
inverse limit has a strong self-similarity: every open subset of~$\hI_s$
contains a homeomorphic copy of $\hI_t$ for every $t\in[\sqrt{2}, 2]$.

In this paper we take an alternative point of view and study the abundance of
 tame behavior. A point $\ux\in\hI_s$ is called \emph{globally leaf regular} if
 its path component is intrinsically homeomorphic to $\R$ (that is, if it is a
 continuous injective image of~$\R$), and this path component is dense and
 metrically infinite in both directions. For certain values of $s$ the
 set~$\GR$ of globally regular points is well understood. For example, if the
 critical point is $n$-periodic then $\GR$ consists of the entire inverse limit
 except for $n$ components which are intrinsically homeomorphic to $[0,
 \infty)$~\cite{BBD, B2}. Here we show that, for all $s$, the globally leaf regular
 points are typical in the inverse limit with respect to a natural measure. In
 contrast, globally leaf regular points are topologically typical only for
 those $s$ for which the critical orbit of $f_s$ is not dense.

\begin{theorem}\label{main}
	Let $s\in(\sqrt{2}, 2]$, and let $\hI_s$ be the inverse limit of the tent
	map $f_s$ restricted to its core~$I_s$.
	\begin{enumerate}[(a)]
        \item For all $s$, the set of globally leaf regular points has full
        measure with respect to the measure induced on $\hI_s$ by the unique
        absolutely continuous invariant measure for $f_s$.

		\item If the critical orbit of $f_s$ is not dense in $I_s$, then the
	 	set of globally leaf regular points of $\hI_s$ contains a dense
	 	$G_\delta$ set.

		\item If the critical orbit of $f_s$ is dense in $I_s$, then the
	 	complement of the set of globally leaf regular points of $\hI_s$
	 	contains a dense $G_\delta$ set. Indeed, there is a dense $G_\delta$
	 	set of points whose path components are either points or are locally
	 	homeomorphic to~$[0,1)$, and whose $\hf_s$-orbits are dense in~$\hI_s$.
	\end{enumerate}
\end{theorem}

Parts~(b) and~(c) of this result are built upon the foundation of well-known
topological properties of tent map inverse limits and, while new, are expected
from the existing literature (see, in particular, \cite{bruin}). Part~(a), on
the other hand, requires the introduction of new measure theoretic ideas: in
particular, a measure $\alpha_x$ defined on almost every fiber $\pi_0^{-1}(x)$,
where $\pi_0\colon\hI_s\to I_s$ is the projection. This measure is defined by
an explicit formula in Definition~\ref{alpha_x}: its crucial property, given by
Theorem~\ref{holo}, is its holonomy invariance on appropriate collections of
arcs in~$\hI_s$ which connect two fibers.  We also show (Theorem~\ref{disthm})
that the collection of measures $\{\alpha_x\}$ is a prescribed scalar multiple
of the disintegration onto the fibers of the measure induced on $\hI_s$ by the
unique absolutely continuous invariant measure for~$f_s$. These properties,
coupled with the topological structure of~$\hI_s$ and the ergodicity of the
natural extension, yield part~(a) of the theorem.

An additional motivation for this work is the discovery in~\cite{prime} that
the natural extensions of tent maps are semi-conjugate to sphere
homeomorphisms, by semi-conjugacies for which all fibers except perhaps one are
finite. For some parameters~$s$ these sphere homeomorphisms are pseudo-Anosov
maps, and for some they are generalized pseudo-Anosovs (as defined
in~\cite{gpA}). Our results here show that for the other parameters there is
still an analog of an invariant unstable foliation, which carries a holonomy
invariant transverse measure. It is important to note that these foliations
have to be understood in a suitable sense. There are always bi-infinite leaves
but there need not be even measurable foliations: we will show in a subsequent
paper that when the critical orbit of $f_s$ is dense there is no measurably
foliated chart in the neighborhood of any point.

In \cite{su}, Su shows that in the inverse limit of a rational map of the
Riemann sphere the typical path component (with respect to a natural measure)
has the affine structure of the complex plane. The measure theoretic results
presented here can be seen as analogs of this result, as we show that, with
respect to a natural measure, the typical path component is intrinsically
isometric to the real line. The results here were in part inspired by Su's
paper and we borrowed several ideas from it, most prominently the use of boxes,
and the main ideas in the proofs of our Lemma~\ref{pccomp} and
Theorem~\ref{main}(a). Note also that Lyubich and Minksy give a deep study of
the inverse limits of rational maps from a somewhat different point of view
in~\cite{lyumin}.

\medskip \medskip

\section{Basic topology and notation}
We consider a fixed tent map $f_s:I\raw I$ with slope $s\in(\sqrt{2}, 2]$ and
critical point $c$, restricted to its core $I_s = [f_s^2(c), f_s(c)]$. Since we
consider a fixed map, we suppress the dependence on $s$ and rescale so that the
core is $[0,1]$. As a result, although there are no subscripts~$s$ in the
remainder of the paper, $f\colon I\to I$ will always denote a core tent map of
slope~$s$ defined on~$I=[0,1]$ (so that $f(c)=1$ and $f^2(c)=0$), and $\hI :=
\ilim(f, I)$ will always denote the inverse limit of this core tent map. 

Points (also called \emph{threads}) in $\hI$ are denoted $\ux = \thrn{x}$, with
$f(x_{i+1}) = x_i$ for all $i\geq 0$. The standard metric on~$\hI$ is given by
\[
	d(\ux, \ux') = \sum_{i=0}^\infty \frac{|x_i-x_i'|}{2^i}\, .
\]
The \emph{projections} are the maps $\pi_n:\hI\raw I$ given by $\pi_n(\ux) =
x_n$, for \mbox{$n\ge 0$}. The \emph{shift} or \emph{natural extension of $f$}
is the homeomorphism $\hf:\hI\raw\hI$ given by $\hf(\ux) = \thr{f(x_0), f(x_1),
f(x_2), \dots} =
\thr{f(x_0), x_0, x_1, \dots}$. Fundamental relations are $f\circ \pi_n
= \pi_n\circ \hf$ and $\pi_{m+n}\circ \hf^n = \pi_m$. If~$K\subseteq \hI$, we
write $K_n := \pi_n(K)$.

Until the end of Section~\ref{typ-topol}, where we complete the proofs of
parts~(b) and~(c) of Theorem~\ref{main}, we will assume that $s\in(\sqrt{2},
2)$, in order to avoid complicating some statements with exceptions. When $s=2$
it is straightforward to show that there is a dense $G_\delta$ set of globally
regular points.

\begin{convention}
	For brevity, we use the term \emph{interval} exclusively to mean a
	non-trivial subinterval of~$I$ (open, half-open, or closed); the term
	\emph{arc} to mean a subset of~$\hI$ which is intrinsically homeomorphic to
	such an interval (and so may be open, half-open, or closed); and the term
	\emph{continuum} to mean a subcontinuum of $\hI$ (which, as usual, must
	contain more than one point).
\end{convention}

The following fact about the dynamics of tent maps is well known.

\begin{lemma}\label{well-known}
	Let~$J$ be an interval. If $f^2(J)\not=I$, then $|f^2(J)|\ge s^2|J|/2$. In
	particular, if $N\ge -2\log(|J|)/\log(s^2/2)$ then $f^N(J)=I$.
\end{lemma}

The analysis of path components has been a central part of the study of tent
map inverse limits (see, for example, \cite{BM2,BB,B2}). The starting point is
this well-known basic characterization~\cite{ingrambook}.

\begin{theorem}\label{basics} 
    The inverse limit $\hI$ contains no subset homeomorphic to a circle or to
     the letter ``Y''. Therefore every path component of~$\hI$ is either a
     point or an arc.
\end{theorem}

\begin{definition}[Locally leaf regular, terminal, and solitary points]
	A point $\ux\in\hI$ is called \emph{locally leaf regular} if its local path
	component is homeomorphic to $(0,1)$, a \emph{terminal point} if its local
	path component is homeomorphic to $[0,1)$ and a \emph{solitary point} if
	its local path component is just itself.
\end{definition}

\begin{definition}[End continuum]
	A continuum~$K$ is an \emph{end continuum} if, whenever $A$ and $B$ are
	continua with $K\subseteq A$ and $K\subseteq B$, then either $A\subseteq B$ or
	$B\subseteq A$.
\end{definition}

The following result is Lemma~7 of~\cite{BBD}.

\begin{theorem}[Barge, Brucks, \& Diamond]\label{BBD}
	If $K$ is a continuum with $0\in K_n$ for infinitely many $n$, then $K$ is
	an end continuum.
\end{theorem}

\begin{remark}\label{critimage} 
	Since $f^{-1}(0) = \{1\}$ and $f^{-1}(1) =\{c\}$, it follows that $0\in
	K_n$ for infinitely many $n$ if and only if $1\in K_n$ for infinitely many
	$n$ if and only if $c\in K_n$ for infinitely many $n$.	
\end{remark}

\section{Countable $0$-\simple decomposition of arcs}
\subsection{\Simple arcs and interval threads}

The next definition formalizes and names a standard tool in the theory of
inverse limits of interval maps. 

\begin{definition}[Interval thread]
	Let $J_0, J_1, J_2, \dots$ be a sequence of intervals, with $f(J_{i+1}) =
	J_i$ for each~$i$. We write
	\[
		\uJ = \thrn{J} := \{\ux\in\hI\,:\, x_i\in J_i \text{ for each } i\} \subseteq \hI,
	\]
	and call~$\uJ$ an \emph{interval thread}. Equivalently, we can write $\uJ =
	\ilim(J_i, f_{|J_i})$.
\end{definition}

Every continuum~$K$ is an interval thread (on a sequence of closed intervals),
since each~$K_n$ is a closed interval and $K=\thrn{K}$.

\begin{definition}[\Simple arc and \simple interval thread]
	 An arc $\gamma$ is \emph{$m$-\simple (over the interval~$J$)} if
	$\pi_{m\vert\gamma}$ is a homeomorphism onto its image $\pi_m(\gamma)=J$ .
	An arc is called \emph{\simple} if it is \simple for some $m\ge0$.

	An interval thread $\thrn{J}$ is \emph{$m$-\simple} if $f$ sends $J_{i+1}$
	homeomorphically onto $J_i$ for all $i\geq m$, or, equivalently, if
	$c\not\in\Intt J_i$ for all $i> m$. An interval thread is called
	\emph{\simple} if it is \simple for some $m\ge0$.
\end{definition}

An arc $\gamma$ (respectively an interval thread $\uJ$) is $m$-\simple if and
only if $\hf^{-m}(\gamma)$ (respectively $\hf^{-m}(\uJ)$) is $0$-\simple.
Therefore many proofs of properties of \simple arcs and interval threads reduce
to the $0$-\simple case.

Any \simple arc $\gamma$ is equal to the interval thread $\thrn{\gamma}$ and any
\simple interval thread is an arc. Moreover, it is easy to check that an arc is
$m$-\simple if and only if it is an $m$-\simple interval thread. Because of this
equivalence we will go back and forth freely between the terminology and
notation of \simple arcs and \simple interval threads.

$0$-\simple arcs are closely related to the \textit{basic arcs} defined
 symbolically in \cite{bruin,jerana}, and elsewhere: a $0$-\simple arc is a
 (non-degenerate) subarc of a basic arc. 

By Theorem~\ref{BBD} and Remark~\ref{critimage}, if $K$ is a continuum but not
 an end continuum, then $c\in K_n$ for only finitely many $n$. We therefore
 have the following corollary:

\begin{corollary}[Brucks and Bruin~\cite{BB}]\label{BBDcor} 
	If $K$ is a continuum but not an end continuum, then it is a \simple closed
	arc.
\end{corollary}

\subsection{The $0$-\simple decomposition}

\begin{definition}[$0$-\simple decomposition, node]
	A \emph{$0$-\simple decomposition} of an arc $\gamma$ is a countable collection of
	$0$-\simple arcs $\gamma\up{i}$ such that
	\begin{enumerate}
		\item $\gamma = \cup \gamma\up{i}$,
		\item $\gamma^{(i+1)}\cap \gamma\up{i}$ is
		a single point, called a \emph{node of the decomposition} and
		denoted~$\uz\up{i}$, and
		\item $\gamma\up{i}\cap \gamma\up{j} = \emptyset$ when $|i-j|>1$.
	\end{enumerate}
	The $0$-\simple decomposition is called \emph{efficient} if whenever
	$\gamma'\subseteq \gamma$ is $0$-\simple, we have $\gamma'\subseteq \gamma\up{i}$
	for some $i$.

	If an arc has an efficient $0$-\simple decomposition, then this
	decomposition is unique and is determined by its nodes. In this case we
	refer to these nodes as \emph{the nodes of the arc}.
\end{definition}
The next lemma is certainly known to experts
but does not seem to be stated in the literature
in  the form we need;
it is implicit in \cite{bruin} and~\cite{jerana}. Rather 
than introduce the symbolic machinery used in those papers we maintain
a strictly topological perspective for brevity of exposition
and self-sufficiency.
\begin{lemma}\label{arcdecom} \ 
	\begin{enumerate}[(a)]
		\item Every \simple arc has a finite efficient $0$-\simple decomposition.
		\item Every closed arc which is contained in an open arc is \simple.
		\item Every open arc has an efficient $0$-\simple decomposition.
	\end{enumerate}
\end{lemma}

\begin{proof}
	For (a), let~$\gamma=\thrn{\gamma}$ be an $m$-\simple arc. Let the closed
	 intervals of monotonicity of $f^m_{\vert\gamma_m}$ be
	 $I\up{1},\dots,I\up{N} \subseteq \gamma_m$, ordered from left to right;
	 and define~$z\up{i}$ by $I\up{i} \cap I\up{i+1} = \{z\up{i}\}$ for $1\le
	 i\le N-1$.

	By assumption, $f_{\vert\gamma_\ell}$ is a homeomorphism for all $\ell>m$.
	Therefore, if~$1\le i\le N$ and $k>0$, there is a unique interval
	$I\up{i}_{m+k} \subseteq \gamma_{m+k}$ for which $f^k\colon I\up{i}_{m+k} \to
	I\up{i}$ is a homeomorphism. So for each such~$i$ there is a $0$-\simple
	interval thread
	\[
		\uI\up{i} = \thr{f^m(I\up{i}), \dots, f(I\up{i}), I\up{i},
			I\up{i}_{m+1}, I\up{i}_{m+2}, \dots} \subseteq \gamma.
	\]

	Similarly, if $1\le i < N$ and $k>0$, there is a unique $z_{m+k}\up{i}\in \gamma_{m+k}$ for which $f^k(z_{m+k}\up{i}) = z\up{i}$, giving threads
	\[
		\uz\up{i} = \thr{f^m(z\up{i}), \dots, f(z\up{i}), z\up{i},
			z\up{i}_{m+1}, z\up{i}_{m+2}, \dots} \in\gamma.
	\]

	It is straightforward to check that the collection of arcs $\uI\up{i}$
	($1\le i\le N$) is a $0$-\simple decomposition of~$\gamma$ with
	nodes~$\uz\up{i}$. Since the nodes are exactly the points $\ux\in\gamma$
	satisfying $x_j=c$ for some $1\le j\le m$, no $0$-\simple subarc of $\gamma$
	can contain a node in its interior. The decomposition is therefore
	efficient.

	(b) follows immediately from Corollary~\ref{BBDcor}, since a closed arc
	which is contained in an open arc cannot be an end continuum.

	For~(c), let $\gamma$ be an open arc, and write $\gamma$ as an increasing
	union \mbox{$\gamma=\bigcup \gamma\up{n}$} of closed (and therefore \simple) arcs.
	Then the union of the nodes of the arcs~$\gamma\up{n}$ determines an
	efficient $0$-\simple decomposition of~$\gamma$.
\end{proof}

\section{Global leaf regularity}

\subsection{The metric on arcs}

\begin{definition}[The metric $\rho$ on an open or \simple arc~$\gamma$]
	\label{arc-metric} 
	Let $\gamma$ be an open arc or a \simple arc, and let $\ux$
	and $\ux'$ be distinct elements of $\gamma$. We define
	\[
		\rho(\ux, \ux') = 
			\sum_{i=0}^{N-1} \left|z_0\up{i} - z_0\up{i+1} \right|,
	\]
	where $\uz\up{0} = \ux$, $\uz\up{N} = \ux'$, and $\uz\up{1}, \dots,
	\uz\up{N-1}$ are the nodes of the efficient $0$-\simple decomposition of the
	(\simple) closed subarc of~$\gamma$ with endpoints $\ux$ and~$\ux'$.
\end{definition}
 
\begin{remark}\label{altmetric} 
	A more standard metric on~$\gamma$ is the intrinsic metric: choose a
	parameterization $\sigma:[0,1]\raw\gamma$ of the subarc with endpoints
	$\ux$ and $\ux'$, and set
	\[
		\beta(\ux, \ux') = \sup \left\{
			\sum_{i=0}^{n-1} d(\sigma(t_i), \sigma(t_{i+1})),
		\right\}
	\]
	where the supremum is over all subdivisions $0=t_0<t_1<\dots<t_n=1$ of
	$[0,1]$. We will show that $\beta(\ux, \ux') = \frac{2s}{2s-1}\,\rho(\ux,
	\ux')$ for all $\ux, \ux'\in\gamma$, so that the two metrics are just
	scaled versions of one another. The use of $\rho$ makes some calculations
	cleaner.

	\medskip

	It is enough to show this in the case where $\gamma$ is $0$-\simple, since
	it is immediate from the definitions that (using the notation of
	Definition~\ref{arc-metric}) we have $\rho(\ux, \ux') = \sum_{i=0}^{N-1}
	\rho(\uz\up{i}, \uz\up{i+1})$ and $\beta(\ux, \ux') = \sum_{i=0}^{N-1}
	\beta(\uz\up{i}, \uz\up{i+1})$.

	Assume, then, that $\gamma$ is $0$-\simple, so that $\rho(\ux, \ux') =
	|x_0-x_0'|$. Since $f^n_{|\gamma_n}$ is a homeomorphism with derivative
	$\pm 1/s^n$ for each $n>0$, we have that $|x_n-x_n'| = |x_0-x_0'|/s^n$, from
	which it follows that $d(\ux, \ux') = \frac{2s}{2s-1}\,\rho(\ux,
	\ux')$. On the other hand, again using that $f^n_{|\gamma_n}$ is a
	homeomorphism for all~$n$, if $\ux''$ lies on the subarc of~$\gamma$ with
	endpoints $\ux$ and $\ux'$, then $x_n''$ lies between $x_n$ and $x_n'$ for
	all~$n$, so that $d(\ux, \ux') = d(\ux, \ux'') + d(\ux'', \ux')$. It
	follows that $\beta(\ux, \ux') = d(\ux, \ux') = \frac{2s}{2s-1}\,\rho(\ux,
	\ux')$ as required.
\end{remark}

Given  a \simple closed arc $\gamma$ with endpoints~$\ux$ and $\ux'$, we write
\mbox{$\rho(\gamma) := \rho(\ux, \ux')$}.

\begin{lemma}
	If~$\gamma$ is a \simple closed arc, then $\rho(\hf(\gamma)) =
	s\rho(\gamma)$.
\end{lemma}

\begin{proof}
	As in Remark~\ref{altmetric}, it suffices to show this when~$\gamma =
	\thrn{\gamma}$ is \mbox{$0$-\simple}. If $c\not\in\Intt\gamma_0$ then
	$\hf(\gamma)$ is also $0$-\simple, and the result follows since
	$|\hf(\ux)_0 - \hf(\ux')_0| = |f(x_0) - f(x'_0)| = s|x_0-x_0'|$ (where
	$\ux$ and $\ux'$ are the endpoints of~$\gamma$).

	On the other hand, if $c\in\Intt\gamma_0$, let $\ux''\in\gamma$ be the
	point with $x''_0=c$: then the efficient $0$-\simple decomposition of
	$\hf(\gamma)$ has node $\hf(\ux'')$, and hence $\rho(\hf(\ux), \hf(\ux')) =
	\rho(\hf(\ux),
	\hf(\ux'')) + \rho(\hf(\ux''), \hf(\ux')) = s\rho(\ux, \ux') +
	s\rho(\ux'', \ux') = s\rho(\ux, \ux')$ as required.
\end{proof}

\begin{corollary}\label{long}
	If $\gamma$ is a \simple closed arc with $\gamma_\ell = I$ for some $\ell\ge
	0$, then $\rho(\gamma)\ge s^\ell$.
\end{corollary}

\begin{proof}
	We have $(\hf^{-\ell}(\gamma))_0 = I$, so that $\rho(\hf^{-\ell}(\gamma))\ge 1$ by Definition~\ref{arc-metric}.
\end{proof}

\subsection{Density} 

Recall that a subset~$K$ of $\hI$ is \emph{$\epsilon$-dense} in~$\hI$ if
$d(\ux, K) < \epsilon$ for all $\ux\in\hI$.

\begin{lemma}\label{density1}
	Let $K$ be a continuum.
	\begin{enumerate}[(a)]
		\item If $K_\ell = I$ for some $\ell>0$, then $K$ is $2^{-\ell}$-dense
		in $\hI$.

		\item Let $|\pi_0(K)|=\delta > 0$. If $N \ge -2\log(\delta) /
		\log(s^2/2)$, then for all $j>0$, $\hf^{N+j}(K)$ is $2^{-j}$-dense in
		$\hI$.
	\end{enumerate}
\end{lemma}

\begin{proof}
	Part (a) is obvious, and~(b) follows from Lemma~\ref{well-known} and~(a).
\end{proof}

\begin{definition}[Metrically infinite]
	A path-connected subset~$S$ of~$\hI$ is \emph{metrically infinite} if, for
	all $N\ge 0$, there is a \simple closed arc~$\gamma\subseteq S$ with
	$\rho(\gamma)>N$.
\end{definition}

\begin{lemma}\label{coollem}
	A path connected subset~$S$ of~$\hI$ is dense in~$\hI$ if and only if for
	every~$\ell\ge 0$ there is a \simple closed arc~$\gamma\subseteq S$ with
	$\gamma_\ell=I$. In this case, $S$ is also metrically infinite.
\end{lemma}

\begin{proof}
	A \simple closed arc~$\gamma$ with $\gamma_\ell=I$ is $2^{-\ell}$-dense
	in~$\hI$ by Lemma~\ref{density1}~(a), which establishes sufficiency of the
	condition. Such an arc has $\rho(\gamma) \ge s^\ell$ by
	Corollary~\ref{long}, so the condition also implies that~$S$ is metrically
	infinite.

	For the converse, suppose that~$S$ is dense in~$\hI$, so that~$S$ is either
	an open arc or a half-open arc. Removing an endpoint in the half-open case,
	we can assume that~$S$ is an open arc, so that there is a continuous
	bijection $\sigma\colon(-1,1) \to S$. For each~$k\ge 2$, let
	$\gamma\up{k}=\sigma([-1+1/k, 1-1/k])\subseteq S$, a \simple closed arc. We
	show that for every~$\ell\ge 0$ there is some~$k$ with
	$\gamma\up{k}_\ell = I$, which will establish the result.

	Suppose for a contradiction that there is some fixed~$\ell$ such that
	$\gamma\up{k}_\ell\not=I$ for all~$k$. By Lemma~\ref{well-known} we have
	$|\gamma\up{k}_{\ell+2}| < 2/s^2 < 1$ for all~$k$. Since
	$\gamma\up{k}_{\ell+2}$ is an increasing sequence of intervals, there is an
	open interval~$J$ which is disjoint from all of
	the~$\gamma\up{k}_{\ell+2}$. Thus $\pi_{\ell+2}^{-1}(J)$ is disjoint from
	the dense set~$S$, which is the required contradiction.
\end{proof}

The converse of the last statement in the lemma is not true in general: there
may be metrically infinite path connected sets which are not dense.

\begin{definition}[Metrically bi-infinite and bi-dense open arcs]
	Let $\gamma$ be an open arc. We say that $\gamma$ is \emph{metrically
	bi-infinite} (respectively \emph{bi-dense}) if for some (and hence for all)
	$p\in\gamma$, both components of $\gamma\setminus\{p\}$ are metrically
	infinite (respectively dense in~$\hI$).
\end{definition}

\begin{lemma}
	If an open arc $\gamma$ is bi-dense then it is metrically bi-infinite, and
	is a path component of~$\hI$.
\end{lemma}

\begin{proof} 
	That~$\gamma$ is metrically bi-infinite follows from Lemma~\ref{coollem}.
	To see that it is a path component of~$\hI$, suppose to the contrary that
	there is some~$q\not\in\gamma$ which is in the path component of~$\gamma$.
	Let $p\in\gamma$. By Theorem~\ref{basics} there is a unique closed arc
	$\Gamma$ in~$\hI$ with endpoints $p$ and $q$. Then~$\Gamma$ contains one of
	the two components of $\gamma\setminus\{p\}$, contradicting the fact that
	these rays are both dense in~$\hI$.
\end{proof}

\subsection{Global leaf regularity}
A point $\ux\in\hI$ is called \emph{globally leaf regular} if its path
component is a bi-dense (and hence metrically bi-infinite) open arc. Let $\GR$
denote the collection of globally leaf regular points,

\begin{lemma}\label{glr} 
	Let $\ux\in\hI$. The following two conditions each imply that $\ux\in\GR$.
	\begin{enumerate}[(a)]
		\item There exists $\epsilon>0$ such that, for arbitrarily large~$n$,
		there is a \simple closed arc~$\gamma$ with (i)
		$\ux\in\Intt(\hf^{n}(\gamma))$; and (ii) each component $T$ of
		$\gamma\setminus\{\hf^{-n}(\ux)\}$ satisfies $|\pi_0(T)| \ge \epsilon$.

		\item There exists $\delta>0$ such that $|x_n-c|\ge\delta$ for all
		$n$.
	\end{enumerate}
\end{lemma}

\begin{proof}
	Suppose that the condition in~(a) holds, so that in particular~$\ux$ is
	locally leaf regular. Let~$C$ be the path component of~$\ux$, and let $S$
	be the union of~$\ux$ with either one of the path components of
	$C\setminus\{\ux\}$. We will show that $S$ is dense in~$\hI$: it follows
	that~$S$ is a half-open arc, and hence that~$C$ is a bi-dense open arc as
	required.

	Let $N > -2\log(\epsilon)/\log(s^2/2)$. Given any~$\ell\ge 0$, pick $n
	\ge \ell + N$ for which a \simple closed arc~$\gamma$ as in~(a) exists.
	Let~$\Gamma$ be the \simple closed arc given by the union of~$\hf^{-n}(\ux)$
	and the component of $\gamma\setminus\{\hf^{-n}(\ux)\}$ which ensures
	$\hf^n(\Gamma) \subseteq S$. Then
	\[
		\pi_{n-N}(\hf^n(\Gamma)) = \pi_0(\hf^N(\Gamma)) 
			= f^N(\pi_0(\Gamma)) = I
	\]
	by Lemma~\ref{well-known}, since $|\pi_0(\Gamma)| \ge \epsilon$. Since
	$n-N\ge\ell$, we have $\pi_\ell(\hf^n(\Gamma)) = I$, and hence~$S$ is dense
	in~$\hI$ by Lemma~\ref{coollem}, as required.

	For~(b), take any $n\ge 0$ and define intervals $\gamma_i =
	[x_{n+i}-\delta/s^i, x_{n+i} + \delta/s^i]$. Then $\gamma = \thrn{\gamma}$
	is a $0$-\simple closed arc which satisfies the conditions of~(a)
	for~$\epsilon = 2\delta$.
\end{proof}

There are many cases in which one can check directly that particular points are
globally leaf regular using these criteria. In the statement below we use the
following notation: 

\begin{notation}[$\hX$]
	If~$X$ is a compact subset of~$I$ with $f(X)=X$, then we write $\hX :=
	\ilim(f_{|X}, X)$. 
\end{notation}

Notice that $\hf(\hX) = \hX$. Examples of such invariant compact subsets are
provided by the orbit closures $\overline{\orb(x)}$ of recurrent points~$x\in
I$.

\begin{corollary}\label{direct}
    If $X\subseteq I$ is compact, with $f(X)=X$ and $c\not\in X$, then each
    $\ux\in\hX$ is globally leaf regular. In particular, if $x$ is recurrent
    and $c\not\in \overline{\orb(x)}$, then each
    $\ux\in\widehat{\overline{\orb(x)}}$ is globally leaf regular.
\end{corollary}

\begin{proof}
	Immediate from Lemma~\ref{glr}~(b) with $\delta = d(X,c) > 0$.
\end{proof}

The simplest examples which satisfy the criterion of Corollary~\ref{direct} are
periodic points~$\ux$ of $\hf$ with $x_n\not=c$ for all~$n$. The collection of
such periodic points is dense in $\hI$.

\section{Typical in topology} \label{typ-topol}
\subsection{Boxes}

\begin{definition}[Boxes] 
	An open (respectively closed) $m$-box $B$ is a union of open (respectively
	closed) arcs, all of which are $m$-\simple over the same open (respectively
	closed) interval $J$. Thus an $m$-box may be written as a union
	\[
		B = \bigcup \gamma^\eta
	\]
	where each $\gamma^\eta$ is an $m$-\simple arc with $\gamma^\eta_m = J$.

	The \emph{maximal $m$-box} $B$ over an interval $J$ is the union of all
	arcs which are $m$-\simple over $J$.
\end{definition}

\begin{remarks} \label{box-remarks}\ 
	\begin{enumerate}[(a)]
		\item Open and closed boxes need not be open and closed subsets
		of~$\hI$.

		\item A subset~$B$ of~$\hI$ is an $m$-box over~$J$ if and only if~$\hf^{-m}(B)$ is a $0$-box over~$J$.

		\item The arcs $\gamma^\eta$ of an open $m$-box are mutually disjoint, whereas those of a closed $m$-box may intersect at their endpoints.

		\item Let~$B$ be an open $0$-box over~$J$. For each $a\in J$, write
		\mbox{$B_a = \pi_0^{-1}(a) \cap B$}. Let $\ux^{a,\eta}$ denote the
		intersection point of $B_a$ and $\gamma^\eta$: thus $x^{a,\eta}_i
		\in \gamma^\eta_i$ for each~$i$. For each~$N$ there is
		some~$\epsilon>0$ such that if $d(\ux^{a,\eta}, \ux^{a,
		\eta'})<\epsilon$ then $x_i^{a,\eta'}\in \gamma_i^\eta$ for $0\le i\le
		N$, so that $\gamma_i^\eta = \gamma_i^{\eta'}$ for $0\le i\le N$. It
		follows that, for each~$b\in J$, the function $\psi_{a,b}\colon B_a\to
		B_b$ defined by $\psi_{a,b}(\ux^{a,\eta}) =  \ux^{b, \eta}$ is a
		homeomorphism, and hence that the function $ \ux^{x, \eta}\mapsto (x,
		\ux^{a,\eta})$ is a homeomorphism $B \to J\times B_a$.
	\end{enumerate}
\end{remarks}

\begin{lemma}\label{boxclosed} 
	The closure in~$\hI$ of a box is a closed box. In particular, the maximal
	box over a closed interval is closed in~$\hI$.
\end{lemma}

\begin{proof}
	By Remark~\ref{box-remarks}~(b) it suffices to consider the case where~$B =
	\bigcup \gamma^\eta$ is a $0$-box over an interval~$J$. Moreover, we can
	assume without loss of generality that $J$ is closed, for if not then
	$\bigcup \Cl(\gamma^\eta) \subseteq \Cl(B)$ is a $0$-box over $\Cl(J)$.

	A $0$-\simple arc~$\gamma \subseteq J\times I^\infty$ over~$J$ is the graph
	of the function \mbox{$F\colon J\to I^\infty$} defined by $F(x_0) =
	(x_1,x_2,\dots)$, where $\thrn{x} \in\gamma$: in other words, $F = \hf^{-1}
	\circ
	\pi_0|_{\gamma}^{-1}$. The function~$F$ is Lipschitz, since if $x_0, x_0'
	\in J$ with $\pi_0|_{\gamma}^{-1}(x_0) = \ux$ and
	$\pi_0|_{\gamma}^{-1}(x_0') = \ux'$, then, as in Remark~\ref{altmetric}, 
	\[
		d(F(x_0), F(x_0')) =\sum_{i=1}^\infty \frac{|x_i-x_i'|}{2^{i-1}} 
						= \sum_{i=1}^\infty \frac{|x_0-x_0'|}{2^{i-1}\,s^i} 
						= \frac{2}{2s-1}\,|x_0-x_0'|.
	\]

	Therefore the $0$-box~$B$ is the union of a collection of graphs of
	uniformly Lipschitz functions. Conversely, the graph of any function $J\to
	I^\infty$ is a \mbox{$0$-\simple} arc over~$J$, provided that it is
	contained in $\hI$, which is guaranteed if it is contained in~$\Cl(B)$. Now
	if $X$ and $Y$ are compact metric spaces, then, by Arzel\`a--Ascoli, the
	closure in $X\times Y$ of any union of graphs of uniformly Lipschitz
	functions $X\raw Y$ is a union of graphs of functions $X\raw Y$. The result
	follows.
\end{proof}

\subsection{Proof of Theorem~\ref{main}~(b) and~(c)}
    For~(b), suppose that the critical orbit of~$f$ is not dense in~$I$, so
    that $Y := I \setminus \overline{\orb(c)} \not=\emptyset$. Let~$J=(a,b)$ be
    a component of~$Y$. If $n\ge 1$ then $f^{-n}(J)$ is a union of components
    of~$Y$, to each of which~$f^n$ restricts to a homeomorphism onto~$J$.
    Therefore~$B =
	\pi_0^{-1}(J)$ is a union of $0$-\simple arcs over~$J$, i.e.\ an open
	$0$-box.

    Since~$f$ is transitive, so also is $\hf$, and hence $\hf^{-1}$. By a
    theorem of Birkhoff (see Theorem 5.8 in \cite{walters}), there is a dense
    $G_\delta$ subset~$Z$ of~$\hI$ consisting of points whose $\hf^{-1}$-orbits
    are dense. We will establish~(b) by showing that $Z\subseteq\GR$.

	Let $\epsilon = (b-a)/4$ and set $J' = (a+\epsilon, b-\epsilon)$ and $B' =
	\pi_0^{-1}(J') \subseteq B$. Let $\ux\in Z$.
	Then, since~$B'$ is open in~$\hI$, there are arbitrarily large integers~$n$
	with $\hf^{-n}(\ux)\in B'$. For each such~$n$, the arc~$\gamma$ of $\Cl(B)$
	to which $\hf^{-n}(\ux)$ belongs satisfies the conditions of
	Lemma~\ref{glr}~(a). Therefore $\ux\in\GR$ as required.

	\medskip
	
    For~(c), suppose that $\overline{\orb(c)} = I$. In this case, $\hI$ is
    nowhere locally the product of a zero-dimensional set and an interval (see
    Proposition~1 of~\cite{bruin} or Theorem~6.4 of~\cite{raines}), so that no
    box contains an open subset of~$\hI$. Let~$\{U_j\}$ be a collection of open
    intervals which form a countable base for the topology of~$I$, and for
    each~$m\ge 0$ let $B_{m,j}$ be the maximal $m$-box over $\Cl(U_j)$. Then
    each~$B_{m,j}$ is closed in~$\hI$ by Lemma~\ref{boxclosed}, and so is
    nowhere dense. Therefore, by Baire's theorem, the complement~$Z$ of
    $\bigcup B_{m,j}$ is dense $G_\delta$.
	
	By Lemma~\ref{arcdecom}~(b), every locally leaf regular point is contained
	in some $B_{m,j}$, so that~$Z$ consists entirely of terminal and solitary
	points. Since the set of points whose~$\hf$-orbits are dense is also dense
	$G_\delta$, the result follows.
\qed 
\begin{remark}
It follows from the proof of Theorem~\ref{main}~(c) that, when the critical
  orbit is dense, the (disjoint) union of the set of solitary points and the
  set of terminal points is dense~$G_\delta$. The interesting question of which
  one of these is dense~$G_\delta$ remains open. Proposition 4.29
  in~\cite{AABC} shows that for $s$ in the dense~$G_\delta$ set of
  parameters identified in~\cite{BBD} the solitary points and terminal points
  are each dense in the core inverse limit $\hI_s$. The situation under the
  weaker hypothesis of a dense critical orbit is unclear and, more generally,
  Problem 9 in~\cite{AABC} asks for conditions on the orbit of the critical
  point that ensure the existence of solitary points.
\end{remark}

\section{Measure preliminaries}

\subsection{Cylinder sets and fibers}

Let $J_0,\dots,J_n$ be intervals with \mbox{$f(J_{i+1}) = J_i$} for each~$i$.
The associated \emph{interval cylinder set} is
\[
	\cyln{J} = \{\ux\in\hI \colon x_i\in J_i \text{ for } 0 \leq i \leq n\}.
\]
Since $ \cyln{J} = \pi_n^{-1}(J_n)$, it is open in~$\hI$ if
$J_n$ is open. The collection of interval cylinder sets for all $n$, with $J_n$
open in $I$ (that is, the collection of all $\pi_n^{-1}(J_n)$) generates both
the topology and the Borel $\sigma$-algebra of $\hI$.

The set $\pi_n^{-1}(x)$ is called the \emph{$\pi_n$-fiber} over $x$. A
$\pi_0$-fiber is sometimes just called a \emph{fiber}. A \emph{point cylinder
set} in the fiber over $y_0$ is
\[
	\cyln{y} = \{\ux\in\hI \colon x_i= y_i \text{ for } 0 \leq i
	\leq n\}.
\]
Note that $\cyln{y} = \pi_n^{-1}(y_n) \subseteq \pi_0^{-1}(y_0) =
\pi_0^{-1}(f^n(y_n))$. The point cylinder set $\cyln{y}$ is open in
$\pi_0^{-1}(y_0)$.

\subsection{Invariant measures}
We now summarize some basic results about the ``physical" measure for tent
maps. This summary includes contributions of several authors, and has been
extended in a variety of directions~\cite{LY,DGP,HG,Ry,baladi}. As before,
$f\colon[0,1]\to[0,1]$ denotes a tent map of fixed slope $s\in(\sqrt{2}, 2]$,
restricted to its core.

\begin{theorem}
	$f$ has a unique invariant Borel probability measure $\mu$ which is
	absolutely continuous with respect to Lebesgue measure $m$, and $d\mu =
    \varphi dm$ with $\varphi\in L^1(m)$ defined on $[0,1]\setminus \orb(c)$.
    The function $\varphi$ can be chosen in its $L^1$-class to be strictly
    positive, of bounded variation, and
	\begin{equation}\label{phipropn}
		\varphi(x) = \sum_{f^n(y)=x} \frac{\varphi(y)}{s^n} 
					\qquad \text{(all $x\not\in\orb(c)$ and $n \ge 0$).}
	\end{equation}
	Finally, $\mu$ is ergodic.
\end{theorem}

Note that if $x\not\in\orb(c)$, so that $\varphi(x)$ is defined, then
$\varphi(y)$ is also defined whenever $f^n(y)=x$. In particular, given a thread
$\thrn{x}$, if $\varphi(x_0)$ is defined, then so is each $\varphi(x_i)$. The
measure $\mu$ is conventionally called the unique \emph{acim} (absolutely
continuous invariant measure) for~$f$. The symbol~$\varphi$ will always denote
the density of this measure.

If $\nu$ is an $f$-invariant Borel probability measure on $I$ then there is a
unique \mbox{$\hf$-invariant} Borel probability measure~$\hnu$ on $\hI$ with
the property that $(\pi_n)_* \hnu = \nu$ for all~$n$~\cite{para}. The measure
$\hnu$ is sometimes called the \textit{inverse limit} or \textit{natural
extension} of the $f$-invariant measure $\nu$. The measure $\hnu$ is
$\hf$-ergodic if and only if $\nu$ is $f$-ergodic. We will be exclusively
concerned with the $\hf$-invariant measure $\hmu$ on $\hI$ derived from the
acim $\mu$ on~$I$.

\subsection{A measure on fibers}
The formalities of Borel measures on fibers are very similar to those on
symbolic subshifts, which is one reason for adopting the language of cylinder
sets. We next define explicitly a measure on fibers, which turns out to be
$\varphi(x)$ times the disintegration of $\hmu$ onto fibers (see
Theorem~\ref{disthm}).

\begin{definition}[The measures $\alpha_x$]
\label{alpha_x}
    For each $x\in I\setminus \orb(c)$ and each point cylinder set in the fiber
    $\pi_0^{-1}(x)$, define
	\begin{equation}\label{alpha}
		\alpha_x(\cyl{x, x_1, \dots, x_n}) = \frac{\varphi(x_n)}{s^n}.
	\end{equation}
\end{definition}

By~\eqref{phipropn}, $\alpha_x$ is finitely additive on the semi-algebra of
point cylinder sets. Exactly as in the case of symbolic subshifts (see
\S 0.2 of~\cite{walters}), $\alpha_x$ extends to the $\sigma$-algebra generated
by the cylinder sets, namely the Borel $\sigma$-algebra of $\pi_0^{-1}(x)$. We
regard each~$\alpha_x$ as a measure on $\hI$ supported on $\pi_0^{-1}(x)$, so
that if $E$ is a Borel subset of~$\hI$ we have $\alpha_x(E) = \alpha_x(E \cap
\pi_0^{-1}(x))$.

\section{Holonomy invariance of $\alpha_x$ in $0$-boxes}

\begin{theorem}\label{holo}
    If $B$ is a $0$-box over $J$ then, for all $a,b\in J\setminus \orb(c)$,
	\[
		\alpha_{a}(B) = \alpha_{b}(B).
	\]
\end{theorem}

\begin{proof}
	Write~$B = \bigcup \gamma^\eta$, where each~$\gamma^\eta = \thr{J,
	\gamma_1^\eta, \gamma_2^\eta, \dots}$ is $0$-\simple over~$J$. For each $n\ge
	1$, let $\cyls{J, J\up{i}_{1,n}, J\up{i}_{2,n}, \dots, J\up{i}_{n,n}}$
	($1\le i\le N(n)$) be the interval cylinder sets which are realized by the
	first~$n+1$ entries of some~$\gamma^\eta$. That is, for each~$i$ there is
	some~$\eta$ with $J\up{i}_{j,n} = \gamma_j^\eta$ for $1\le j\le n$, and each~$\eta$ arises in this way. Then, for each~$n$,
	\[
		B \subseteq \bigcup_{i=1}^{N(n)}
			\cyl{J, J\up{i}_{1,n}, J\up{i}_{2,n}, \dots, J\up{i}_{n,n}},
	\]
    and the sets in this union are mutually disjoint except perhaps along the
    fibers of endpoints of~$J$, if those endpoints lie in~$\orb(c)$. Moreover,
	\[
		B = \bigcap_{n=1}^\infty \bigcup_{i=1}^{N(n)}
			\cyl{J, J\up{i}_{1,n}, J\up{i}_{2,n}, \dots, J\up{i}_{n,n}}.
	\]

    Now let $a,b\in J\setminus\orb(c)$. Since $f^j\colon J\up{i}_{j,n} \to J$
    is a homeomorphism for each $i$, $j$, and~$n$, there is a unique point
    $a\up{i}_{j,n} \in J\up{i}_{j,n}$ with $f^j(a\up{i}_{j,n}) = a$. Therefore
	\[
		B \cap \pi_0^{-1}(a) = \bigcap_{n=1}^\infty \bigsqcup_{i=1}^{N(n)}
			\cyl{a, a\up{i}_{1,n}, a\up{i}_{2,n}, \dots, a\up{i}_{n,n}}.
	\]
	Since $\alpha_a$ is a regular measure,~\eqref{alpha} gives
	\[
		\alpha_a(B) = \lim_{n\to\infty} \sum_{i=1}^{N(n)} 
							\frac{\varphi(a\up{i}_{n,n})}{s^n}
		\quad \text{ and analogously } \quad
		\alpha_b(B) = \lim_{n\to\infty} \sum_{i=1}^{N(n)} 
							\frac{\varphi(b\up{i}_{n,n})}{s^n}.
	\]

	For each $n$ and $i$, the points $a_{n,n}\up{i}$ and $b_{n,n}\up{i}$ are
	both in the interval $J_{n,n}\up{i}$. As $i$ varies, the intervals
	$J_{n,n}\up{i}$ are disjoint except perhaps at their endpoints. Recalling
	that $\varphi$ is of bounded variation, let~$V<\infty$ be its total
	variation. Then
	\[
		\begin{split}
		|\alpha_a(B) - \alpha_b(B)| &\leq
		 	\lim_{n\raw\infty} \sum_{i=1}^{N(n)} 
		 	  \frac{|\varphi(a_{n,n}\up{i}) - \varphi(b_{n,n}\up{i})|}{s^n} \\
							&\leq \lim_{n\raw\infty}\frac{V}{s^n} = 0.
		\end{split}
	\]
\end{proof}

\section{Typical in measure}
\subsection{Disintegration of the measure $\hmu$}

The fibers $\{\pi_0^{-1}(x)\}$ provide a measurable partition of $\hI$. Thus,
by Rokhlin's disintegration theorem, there is a family of probability measures
$\{\hmu_x\}$, defined for $\mu$-a.e.\ $x\in I$, with $\hmu_x$ supported on the
fiber $\pi_0^{-1}(x)$, having the property that for any Borel subset $E$ of
$\hI$,
\begin{equation}\label{disint}
	\hmu(E) = \int_I \hmu_x(E) \; d\mu(x).
\end{equation}
Note that $\hmu_x(E) = \hmu_x(E\cap \pi_0^{-1}(x))$, since each $\hmu_x$ is
supported on the fiber $\pi_0^{-1}(x)$. The measures $\hmu_x$ are called the
disintegrations of $\hmu$ onto fibers, or alternatively the conditional
measures of $\hmu$ on fibers. We next show that these conditional measures are
simple multiples of the measures~$\alpha_x$. In this statement, and in the
remainder of the paper, ``almost every'' means with respect to $\mu$ or,
equivalently, with respect to Lebesgue measure~$m$.

\begin{theorem}\label{disthm}
	 $d\alpha_x = \varphi(x)\;d\hmu_x$ for a.e.~$x\in I$. In particular, for
	 any Borel subset~$E$ of~$\hI$,
	\[
		\hmu(E) = \int_I \alpha_x(E) \; dm(x).
	\]
\end{theorem}

\begin{proof}
	 It suffices to show that for a.e.~$x\in I$ we have 
	 \[
		 \alpha_x(\cyl{x, x_1, \dots, x_n}) = 
		 	\varphi(x)\; \hmu_x(\cyl{x, x_1, \dots, x_n})
	\] for each point
    cylinder set $\cyl{x, x_1,\dots, x_n}$ in $\pi_0^{-1}(x)$. Since $\orb(c)$
    is countable, we can assume that $x\not\in\orb(c)$, so that $x_i\not=c$ for
    all~$i$. There is therefore some $\epsilon_0$ with the property that, for
    all $\epsilon<\epsilon_0$, the restriction of $f^n$ to $J_\epsilon =
    [x_n-\epsilon, x_n+\epsilon]$ is a homeomorphism onto its image $I_\epsilon
    = [x-s^n\epsilon, x + s^n\epsilon]$.

	Write $K_\epsilon = \pi_n^{-1}(J_\epsilon)$, so that $\pi_0(K_\epsilon) =
	I_\epsilon$. By~\eqref{disint},
	\[
		\hmu(K_\epsilon) = \int_{I_\epsilon} \hmu_y(K_\epsilon)\; d\mu(y)
						 = \int_{I_\epsilon} \hmu_y(K_{\epsilon_0})\; d\mu(y)
	\]
	since $\hmu_y(K_\epsilon) = \hmu_y(K_{\epsilon_0})$ for $y\in I_\epsilon$.
	By the Lebesgue differentiation theorem, for a.e.~$x\in I$,

	\[
		 \lim_{\epsilon\raw 0}
		 	\frac{\hmu(K_\epsilon)}{\mu(I_{\epsilon})}
			= \hmu_x(K_{\epsilon_0}) 
			= \hmu_x(K_{\epsilon_0}\cap\pi_0^{-1}(x)) 
			= \hmu_x([x, x_1, \dots, x_n]).
	\]
	Since $d\mu = \varphi\,dm$ we have $\lim_{\epsilon\to
	0}\mu(I_\epsilon)/m(I_\epsilon) = \varphi(x)$ for a.e.~$x\in I$, so that
	\[
		\varphi(x)\,\hmu_x([x, x_1, \dots, x_n]) = 
			\lim_{\epsilon\to 0} \frac{\hmu(K_\epsilon)}{m(I_\epsilon)}
			\quad \text{ for a.e.\ } x \in I,
	\]
	and it only remains to show that $\lim_{\epsilon\to 0} \hmu(K_\epsilon) /
	m(I_\epsilon) = \alpha_x(\cyl{x, x_1, \dots, x_n})$ for a.e.~$x\in I$.

	To show this, let~$g = (f^n_{\vert J_\epsilon})^{-1}\colon
	I_\epsilon\to J_\epsilon$ (so that~$g$ has constant slope $\pm 1/s^n$).
	Observing that $\hmu(K_\epsilon) = \mu(J_\epsilon)$ (since $\mu =
	(\pi_n)_*\hmu$), we have that for a.e.~$x\in I$,
	\begin{equation*}
		\begin{split}
			\lim_{\epsilon\to 0} \hmu(K_\epsilon)/m(I_\epsilon) 
			 	&=
			\lim_{\epsilon\to 0} \mu(J_\epsilon)/m(I_\epsilon)\\
			 	&=
			\lim_{\epsilon\to 0} 
				\frac{1}{m(I_\epsilon)} \int_{J_\epsilon} \varphi(y)\; dm(y)\\
				&= 
			\lim_{\epsilon\to 0} 
				\frac{1}{m(I_\epsilon)} \int_{I_\epsilon} 
						\varphi(g(u))\,|g'(u)|\; dm(u)\\
				&= 
			\frac{\varphi(x_n)}{s^n} 
				= 
			\alpha_x([x, x_1, \dots, x_n])
		\end{split}
	\end{equation*}
	as required, using~\eqref{alpha} and the Lebesgue differentiation theorem.
\end{proof}

The important consequence of this result, together with Theorem~\ref{holo},
 for what follows is that the
restriction of $\hmu$ to an open $0$-box is a product:

\begin{corollary}\label{product}
    Let~$B$ be an open $0$-box over an interval~$J$ and take any $a\in
    J\setminus\orb(c)$. Under the homeomorphism $B\to J\times (\pi_0^{-1}(a)
    \cap B)$ defined in Remark~\ref{box-remarks}~(d), the restriction of $\hmu$
    to~$B$ pushes forward to $m \times \alpha_a$. In particular, $\hmu(B) =
    m(J)\alpha_a(B)$.
\end{corollary}

\subsection{Positive measure boxes}

\begin{lemma}\label{pccomp}
    Let $M = \sup\{\varphi(x)\colon x\in I\setminus\orb(c)\}$. For all $N>1$
    there exists an open $0$-box $B$ over an interval $J$ such that, for all
    $x\in J\setminus\orb(c)$,
    \begin{equation}\label{boxsize1}
        \alpha_x(B) \ge M\left(1 - \frac{1}{s^{N-1}}\right),
    \end{equation}
    and in particular
    \begin{equation}\label{boxsize2}
        \hmu(B) \ge M\left(1- \frac{1}{s^{N-1}}\right) m(J)> 0.
    \end{equation}
\end{lemma}

\begin{proof}
Fix~$N>1$, and let~$J$ be a component of $I\setminus\{c, f(c), \dots, f^N(c)\}$
with $\sup\{\varphi(x)\,:\,x\in J\setminus\orb(c)\} = M$.

Let $S=\{m\in\N\,:\,f^m(c)\in J\}$. For each~$m\in S$, let $K_m$ be the
component of $f^{-m}(J)$ containing~$c$, and define
\[
    C_m = \cyl{f^m(K_m),\dots,K_m} \subset \pi_0^{-1}(J).
\]

Now suppose that $\ux = \thrn{x}\in \pi_0^{-1}(J)\setminus\bigcup_{m\in S}
C_m$. Define intervals $J_m\subset f^{-m}(J)$ for $m\ge 0$ inductively by
$J_0=J$, and $J_{m+1}$ is the component of $f^{-1}(J_m)$ which contains
$x_{m+1}$. Then $c\not\in J_m$ for all~$m$, for if $c\in J_m$ we would have
$J_m\subset K_m$ and hence $\ux\in C_m$. It follows that $0\not\in J_m$ for
all~$m$, and hence $f(J_{m+1})=J_m$ for all~$m$. Therefore $\uJ=\thr{J, J_1,
J_2, \dots}$ is a $0$-flat arc over~$J$ which contains~$\ux$. Hence
\[
    B = \pi_0^{-1}(J_0) \setminus\bigcup_{m\in S}C_m
\]
is the maximal $0$-box over~$J$.

If $c\in f^i(K_m)$ for some $m\in S$ and $1\le i < m$, then $m-i\in S$,
$f^i(K_m)\subset K_{m-i}$, and $C_m\subset C_{m-i}$. Defining 
\[
	T=\{m\in S\,:\, c\not\in f^i(K_m) \text{ for } 1\le i < m\},
\]
we therefore have $B=\pi_0^{-1}(J_0) \setminus\bigcup_{m\in T}C_m$.

Let $x\in J\setminus\orb(c)$, and write $T_x=\{m\in T\,:\,x\in
f^m(K_m)\}\subset T$. For each $m\in T_x$, and each $1\le i < m$, let
$x_i\up{m}$ be the unique point of $f^{m-i}(K_m)$ with $f^i(x_i\up{m}) = x$.
Then $\pi_0^{-1}(x) \cap C_m \subset \cyls{x, x_1\up{m},
    \dots, x_{m-1}\up{m}}$, and hence
    \[
        \pi_0^{-1}(x) \cap B \supset \pi_0^{-1}(x) \setminus \bigcup_{m\in T_x}\cyl{x, x_1\up{m}, \dots, x_{m-1}\up{m}}.
    \]
    Since $\alpha_x(\cyls{x, x_1\up{m}, \dots, x_{m-1}\up{m}}) =
    \varphi(x_{m-1}\up{m})/s^{m-1}$, and $m>N$ for all $m\in T_x$ by choice of~$J$, we have
    \[
    \begin{split}
        \alpha_x(B) &\geq \alpha_x(\pi_0^{-1}(x)) -
                    \sum_{m\in T_x} \frac{\varphi(x_{m-1}\up{m})}{s^{m-1}}\\
                    &\geq \varphi(x) - \frac{M}{s^{N-1}}.
    \end{split}
    \]
    By Theorem~\ref{holo}, 
    $\alpha_x(B)$ is independent of $x\in J\setminus\orb(c)$.
    \eqref{boxsize1} therefore holds since~$J$ was chosen so that $M
    =\sup\{\varphi(x): x\in J\setminus\orb(c)\}$, and \eqref{boxsize2} follows
    by Corollary~\ref{product}.
\end{proof}

\subsection{Proof of Theorem~\ref{main}~(a)}
The proof is almost identical to that of Theorem~\ref{main}~(b), using
ergodicity rather than transitivity of~$\hf^{-1}$.

	By Lemma~\ref{pccomp}, there is a $0$-box~$B$ over an interval~$J=(a,b)$
	with $\hmu(B)>0$. Let~$\epsilon = (b-a)/4$, and set $J' = (a+\epsilon,
	b-\epsilon)$ and $B' = \pi_0^{-1}(J') \cap B$. By Corollary~\ref{product},
	$\hmu(B') = \hmu(B)/2>0$.

	Since~$\hf^{-1}$ is ergodic with respect to~$\hmu$, there is a full
	$\hmu$-measure subset~$Z$ of $\hI$ with the property that, for each~$\ux\in
	Z$, there are arbitrarily large integers~$n$ with $\hf^{-n}(\ux)\in B'$.
	For each such~$n$, the arc~$\gamma$ of $\Cl(B)$ to which $\hf^{-n}(\ux)$
	belongs satisfies the conditions of Lemma~\ref{glr}~(a). Therefore
	\mbox{$Z\subseteq\GR$}.
\qed

\smallskip
\noindent\textbf{Acknowledgments:} we would like to thank
Ana Anu\v{s}i\'{c} and Jernej \v{C}in\v{c} for useful conversations, and the
anonymous referee for several helpful comments.

AdC was partially supported by CAPES grant 88881.119100/2016-01.

\bibliographystyle{amsplain}
\bibliography{typ_f}

\providecommand{\bysame}{\leavevmode\hbox to3em{\hrulefill}\thinspace}
\providecommand{\MR}{\relax\ifhmode\unskip\space\fi MR }
\providecommand{\MRhref}[2]{%
  \href{http://www.ams.org/mathscinet-getitem?mr=#1}{#2}
}
\providecommand{\href}[2]{#2}
\begin{thebibliography}{10}

\bibitem{AABC}
L.~Alvin, A.~Anu{\v{s}}i{\'{c}}, H.~Bruin, and J.~{\v{C}}in{\v{c}},
  \emph{Folding points of unimodal inverse limit spaces}, Nonlinearity
  \textbf{33} (2019), no.~1, 224--248.

\bibitem{jerana}
A.~Anu\v{s}i\'c, H.~Bruin, and J.~\v{C}in\v{c}, \emph{Uncountably many planar
  embeddings of unimodal inverse limit spaces}, Discrete Contin. Dyn. Syst.
  \textbf{37} (2017), no.~5, 2285--2300.

\bibitem{baladi}
V.~Baladi, \emph{Positive transfer operators and decay of correlations},
  Advanced Series in Nonlinear Dynamics, vol.~16, World Scientific Publishing
  Co., Inc., River Edge, NJ, 2000.

\bibitem{BBD}
M.~Barge, K.~Brucks, and B.~Diamond, \emph{Self-similarity in inverse limit
  spaces of the tent family}, Proc. Amer. Math. Soc. \textbf{124} (1996),
  no.~11, 3563--3570.

\bibitem{Ingram}
M.~Barge, H.~Bruin, and S.~\v{S}timac, \emph{The {I}ngram conjecture}, Geom.
  Topol. \textbf{16} (2012), no.~4, 2481--2516.

\bibitem{bargemartin}
M.~Barge and J.~Martin, \emph{The construction of global attractors}, Proc.
  Amer. Math. Soc. \textbf{110} (1990), no.~2, 523--525.

\bibitem{BM2}
\bysame, \emph{Endpoints of inverse limit spaces and dynamics}, Continua
  ({C}incinnati, {OH}, 1994), Lecture Notes in Pure and Appl. Math., vol. 170,
  Dekker, New York, 1995, pp.~165--182.

\bibitem{prime}
P.~Boyland, A.~de~Carvalho, and T.~Hall, \emph{{Natural extensions of unimodal
  maps: prime ends of planar embeddings and semi-conjugacy to sphere
  homeomorphisms}}, arXiv:1704.06624 [math.DS] (2017).

\bibitem{BB}
K.~Brucks and H.~Bruin, \emph{Subcontinua of inverse limit spaces of unimodal
  maps}, Fund. Math. \textbf{160} (1999), no.~3, 219--246.

\bibitem{bruin}
H.~Bruin, \emph{Planar embeddings of inverse limit spaces of unimodal maps},
  Topology Appl. \textbf{96} (1999), no.~3, 191--208.

\bibitem{B2}
\bysame, \emph{Asymptotic arc-components of unimodal inverse limit spaces},
  Topology Appl. \textbf{152} (2005), no.~3, 182--200.

\bibitem{gpA}
A.~de~Carvalho and T.~Hall, \emph{Unimodal generalized pseudo-{A}nosov maps},
  Geom. Topol. \textbf{8} (2004), 1127--1188 (electronic).

\bibitem{DGP}
B.~Derrida, A.~Gervois, and Y.~Pomeau, \emph{Iteration of endomorphisms on the
  real axis and representation of numbers}, Ann. Inst. H. Poincar\'e Sect. A
  (N.S.) \textbf{29} (1978), no.~3, 305--356.

\bibitem{HG}
F.~Hofbauer and G.~Keller, \emph{Ergodic properties of invariant measures for
  piecewise monotonic transformations}, Math. Z. \textbf{180} (1982), no.~1,
  119--140.

\bibitem{ingrambook}
W.~Ingram and W.~Mahavier, \emph{Inverse limits: From continua to chaos},
  Developments in Mathematics, vol.~25, Springer-Verlag, 2011.

\bibitem{LY}
A.~Lasota and J.~Yorke, \emph{On the existence of invariant measures for
  piecewise monotonic transformations}, Trans. Amer. Math. Soc. \textbf{186}
  (1973), 481--488 (1974).

\bibitem{lyumin}
M.~Lyubich and Y.~Minsky, \emph{Laminations in holomorphic dynamics}, J.
  Differential Geom. \textbf{47} (1997), no.~1, 17--94.

\bibitem{mis}
M.~Misiurewicz, \emph{Embedding inverse limits of interval maps as attractors},
  Fund. Math. \textbf{125} (1985), no.~1, 23--40.

\bibitem{para}
K.~Parthasarathy, \emph{Probability measures on metric spaces}, AMS Chelsea
  Publishing, Providence, RI, 2005, Reprint of the 1967 original.

\bibitem{raines}
B.~Raines, \emph{Inhomogeneities in non-hyperbolic one-dimensional invariant
  sets}, Fund. Math. \textbf{182} (2004), no.~3, 241--268.

\bibitem{Ry}
M.~Rychlik, \emph{Bounded variation and invariant measures}, Studia Math.
  \textbf{76} (1983), no.~1, 69--80.

\bibitem{su}
M.~Su, \emph{Measured solenoidal {R}iemann surfaces and holomorphic dynamics},
  J. Differential Geom. \textbf{47} (1997), no.~1, 170--195.

\bibitem{walters}
P.~Walters, \emph{An introduction to ergodic theory}, Graduate Texts in
  Mathematics, vol.~79, Springer-Verlag, New York-Berlin, 1982.

\end{thebibliography}

\end{document}